\title{\textsc{\textbf{{
    Harmonic spinors in the Ricci flow
}}}}
\author{\textsc{
    Julius Baldauf\thanks{
    Supported in part by the National Science Foundation. {\it E-mail}: \texttt{juliusbl@mit.edu}
}}
\vspace{0.2cm}\\
    \textsc{\footnotesize MIT Department of Mathematics}\vspace{-0.1cm}\\
    \textsc{\footnotesize Cambridge, MA}
    \vspace{-0.05cm}
}
\date{}

\documentclass[11pt, letterpaper, leqno, twoside]{article}

\usepackage[colorinlistoftodos,prependcaption]{todonotes}
\usepackage{amsmath}
\usepackage{amsthm,xcolor}
\usepackage[colorlinks, linkcolor=red]{hyperref}
\usepackage[top=1.5in, lmargin=1in, rmargin=1in, marginparwidth=0.9in]{geometry}
\usepackage[english]{babel}
\usepackage{amssymb}
\usepackage{mathrsfs}
\usepackage{mathtools}
\usepackage{bbm}
\usepackage{amsthm}
\usepackage{fancyhdr}
\usepackage{tikz-cd}
\usepackage{comment}
\usepackage{etoolbox}
\usepackage{breqn}
\usepackage{esint}
\usepackage{appendix}
\hypersetup{
    colorlinks=true,
    linkcolor=blue,
    filecolor=magenta,      
    urlcolor=cyan,
    citecolor=blue
}
\usepackage[T1]{fontenc}
\usepackage[center]{titlesec}
\usepackage{xcolor}
\usepackage[bottom]{footmisc}
\usepackage{indentfirst}
\usepackage{xpatch}
\usepackage{chngcntr}

\makeatletter
\renewcommand\th@plain{\slshape}
\xpatchcmd{\proof}{\itshape}{\slshape}{}{}
\makeatother

\makeatletter
\renewcommand\th@plain{\slshape}
\makeatother

\titlelabel{\thetitle.\quad}
\titleformat*{\section}{\centering\large\scshape\sffamily}
\titleformat{\subsection}[runin]
  {\normalfont\bfseries}{\thesubsection.}{0.6em}{}
\titleformat{\subsubsection}[runin]
  {\normalfont\bfseries}{\thesubsubsection.}{0.6em}{}
\fancypagestyle{mypagestyle}{%
    \fancyhf{}% Clear header/footer
    
    \fancyhead[OC]{\small 
    Julius Baldauf
    }% Author on Odd page, Centred
    \fancyhead[EC]{\small 
    Harmonic spinors in the Ricci flow
    }% Title on Even page, Centred
    \fancyhead[EL]{\thepage}% Page number on even page, centered
    \fancyhead[OR]{\thepage}% Page number on odd page, centered
    \setlength{\headheight}{13.6pt}
}
\numberwithin{equation}{section}

\theoremstyle{plain} 
\newtheorem{lemma}[equation]{Lemma}

\newtheorem{proposition}[equation]{Proposition}
\newtheorem{theorem}[equation]{Theorem}
\newtheorem{corollary}[equation]{Corollary}

\theoremstyle{definition}

\newtheorem{remark}[equation]{Remark}
\newtheorem{example}[equation]{Example}

\newcommand{\R}{\mathbb{R}}
\newcommand{\Z}{\mathbb{Z}}

\newcommand{\C}{\mathbb{C}}

\newcommand{\id}{\mathbbm{1}}
\renewcommand{\ker}[1]{\mathrm{ker} \left( #1 \right)}

\renewcommand{\tilde}{\widetilde}

\newcommand{\Rm}{\mathrm{Rm}}
\newcommand{\Ric}{\mathrm{Ric}}
\newcommand{\trRic}{\Ric^{\circ}}
\newcommand{\Scal}{\mathrm{R}}

\newcommand{\minScal}{\Scal_{\mathrm{min}}}
\newcommand{\avgScal}{\bar{\Scal}}
\newcommand{\Vol}{\mathrm{Vol}}

\renewcommand{\phi}{\varphi}
\renewcommand{\epsilon}{\varepsilon}

\newcommand{\CP}{\mathbb{CP}}

\newcommand{\Rea}{\,\mathrm{Re}\,}

\newcommand{\Hom}{\mathrm{Hom}}
\newcommand{\End}{\mathrm{End}}

\newcommand{\Hess}{\mathrm{Hess}}

\newcommand{\dL}{\Delta}

\newcommand{\ind}{\mathrm{ind}}

\newcommand{\SO}{\mathrm{SO}}
\newcommand{\Un}{\mathrm{U}}
\newcommand{\SU}{\mathrm{SU}}
\newcommand{\Spin}{\mathrm{Spin}}

\usepackage{microtype}
\begin{document}

\maketitle
%%%%%%%%%%%%%%%%%%%%%%%%%%%%%%%%%%
\begin{abstract}
This paper studies the Ricci flow on closed manifolds admitting harmonic spinors. 
It is shown that Perelman's Ricci flow entropy can be expressed in terms of the energy of harmonic spinors in all dimensions, and in four dimensions, in terms of the energy of Seiberg-Witten monopoles. 
Consequently, Ricci flow is the gradient flow of these energies.
The proof relies on a weighted version of the monopole equations, introduced here.
Further, a sharp parabolic Hitchin-Thorpe inequality for simply-connected, spin 4-manifolds is proven.
From this, it follows that the normalized Ricci flow on any exotic K3 surface must become singular.
\end{abstract}
%%%%%%%%%%%%%%%%%%%%%%%%%%%%%%%%%%
%%%%%%%%%%%%%%%%%%%%%%%%%%%%%%%%%%
\section{Introduction}

Ricci flow is an evolution equation that improves the geometry of a manifold in time \cite{Ha82}. In particular, the flow retains symmetries. 
This property forms the foundation of the theory of K\"ahler-Ricci flow, which has a variety of special properties not enjoyed by Ricci flow on general manifolds.
The K\"ahler structure allows for the use of complex-geometric tools in the analysis of singularities, which has led to a fruitful theory of K\"ahler-Ricci flow, particularly in low dimensions.

Similarly, on spin manifolds, Ricci flow preserves the spin structure, allowing for the use of spin-geometric tools to analyze the flow.
However, while the existence of a K\"ahler structure is a strong topological condition, spin structures exist in abundance: all low-dimensional manifolds admit them. 
Thus, the tools of spin geometry can be applied to all Ricci flows in low dimensions.
The study of spinors has proven to be of fundamental importance in the understanding of static manifolds, and this paper demonstrates that they must also play a fundamental role in the Ricci flow on low-dimensional manifolds, particularly of dimension four. Indeed, Ricci flow can be \textit{defined} as the gradient flow of the energy of a harmonic spinor, and in four dimensions, such spinors are guaranteed to exist on the most interesting manifolds -- those admitting exotic smooth structures.

The gradient flow structure of Ricci flow on general (not necessarily spin) manifolds was discovered by Perelman \cite{P02}. 
The entropy he discovered, now called Perelman's $\lambda$-entropy, can be interpreted as a weighted Hilbert-Einstein energy and is defined not only on the space of metrics, but on the space of pairs $(g,f)$ consisting of a Riemannian metric $g$ and a smooth function $f$ defining the measure $e^{-f}dV$.
While the Hilbert-Einstein functional is defined as the total scalar curvature, Perelman's $\lambda$-entropy minimizes the total weighted scalar curvature
\begin{equation}\label{eqn: weighted scalar}
    \lambda(g)=\inf_f\int_M\Scal_f\,e^{-f}dV, 
    \qquad \qquad 
    \Scal_f = \Scal +2\Delta f -|\nabla f|^2.
\end{equation}
The infimum is taken over all functions $f$ such that $\int_Me^{-f}dV_g=1$, and is achieved by the unique function for which the weighted scalar curvature $\Scal_f$ is constant equal to $\lambda(g)$. 

\subsection{Harmonic spinors}
Remarkably, Perelman's $\lambda$-entropy can also be interpreted as the first eigenvalue of the Schr\"odinger operator $-4\Delta +\Scal$.
It is the latter definition which reveals the relationship to spin geometry. 
To describe this relationship, the central object needed is the weighted Dirac operator determined by a weighted spin manifold $(M,g,e^{-f}dV)$.
This operator is defined as a conjugation of the standard Dirac operator by $D_f=e^{f/2}De^{-f/2}$.
The relationship to the weighted scalar curvature is then established via the associated Schr\"odinger-Lichnerowicz formula
\begin{equation}
    D_f^2 = -\Delta_f +\frac{1}{4}\Scal_f,
\end{equation}
where $\Delta_f$ denotes the weighted Laplacian.
First observed by Perelman, this formula motivated his introduction of the $\lambda$-entropy \cite[1.3]{P02}.
The formula is of particular interest when $\Scal_f$ is constant, which can always be achieved by choosing $f$ to be the minimizer of the $\lambda$-entropy. 
In this case, applied to a spinor $\psi$ lying in the kernel of $D_f$ (which is nonempty if the $\hat{A}$-genus is nontrivial) implies the following formula for the $\lambda$-entropy (Theorem \ref{thm: lambda-entropy in terms of harmonic spinor})
\begin{equation}\label{eqn: lambda-entropy in terms of harmonic spinor intro}
    \lambda(g)=-4\frac{\int_M |\nabla \psi|^2 \,e^{-f}dV}{\int_M|\psi|^2 e^{-f}dV}.
\end{equation}
Along the Ricci flow, Perelman's monotonicity formula therefore implies that the above energy of the spinor is monotone decreasing in time.
It follows immediately from (\ref{eqn: lambda-entropy in terms of harmonic spinor intro}) that manifolds admitting a parallel spinor and with nontrivial Dirac index, for example the K3 surface, are linearly stable \cite{CHI05, DWW05}; moreover, it strengthens the picture that a Ricci flow starting near such a metric exists for all time and the harmonic spinors converge to nontrivial parallel spinors \cite{Se06}.
The monotonicity of (\ref{eqn: lambda-entropy in terms of harmonic spinor intro}) along Ricci flow may also be compared to the spinorial gradient flow introduced in \cite{AWW16}.

\subsection{Monopoles}
The second part of this paper concerns the Ricci flow on 4-manifolds with nontrivial Seiberg-Witten invariant.
An orientable manifold of dimension four does not always admit a spin structure, but it always admits spin-c structures.
To define the Dirac operator on a spin-c manifold, a choice of an auxiliary connection on an associated line bundle is necessary.
The Seiberg-Witten monopole equations determine a special class of such connections \cite{W94}:
this nonlinear elliptic system couples a spinor and a $U(1)$-connection, requiring that the self-dual part of the curvature of the connection is quadratic in the spinor, and that the spinor is harmonic with respect to the Dirac operator defined by said connection.
Solutions of the monopole equations always exist under suitable (smooth) topological conditions, such as the existence of a symplectic structure \cite{T00}.

This paper introduces a weighted version of the monopole equations, which is naturally adapted to the Ricci flow. 
Given a weighted spin-c 4-manifold $(M^4,g,f)$, equipped with a $U(1)$-connection $A$, the associated weighted Dirac operator is defined as a conjugation of the standard spin-c Dirac operator by $D_{A,f}=e^{f/2}D_Ae^{-f/2}$.
In particular, if $(A,\phi)$ is a solution of the monopole equations, then the pair $(A, \psi)$, with $\psi=e^{f/2}\phi$, solves the weighted monopole equations
\begin{align}
        F_A^+\cdot &= e^{-f}(\psi \otimes \psi^*)^{\circ} \\
        D_{A,f}\psi&=0. \nonumber
\end{align}
In the context of Ricci flow, these weighted equations are of particular interest when the weight $f$ is chosen such that $\Scal_f$ is constant.
The associated weighted Schr\"odinger-Lichnerowicz formula then implies the following formula for the $\lambda$-entropy (Theorem \ref{thm: lambda-entropy in terms of monopole})
\begin{equation}
    \lambda(g)
    =-\frac{\int_M\left(4|\nabla^A \psi|^2+e^{-f}|\psi|^4\right)\,e^{-f}dV}{\int_M|\psi|^2\,e^{-f}dV}.
\end{equation}
The process of choosing a weight function $f$ to achieve constant weighted scalar curvature should be viewed as an analogue of conformally rescaling the metric to achieve constant scalar curvature. However, the advantage of modifying the weight over conformally changing the metric is that the geometry remains unchanged by the former transformation. 

The above formula implies a parabolic version of LeBrun's scalar curvature estimate \cite{L95}: if $(M^4,g)$ is a spin-c 4-manifold admitting a monopole class $\mathfrak{s}$ such that $c_1^2(\mathfrak{s})>0$, and if $f$ is the minimizer of Perelman's $\lambda$-entropy normalized so that $\int_M e^{-f}dV=\Vol(M)$, then (Corollary \ref{cor: weighted LeBrun estimate}) 
\begin{equation}\label{eqn: weighted L2 intro}
    \int_M\Scal_f^2\,e^{-f}dV
    \geq 32\pi^2 c_1^2(\mathfrak{s}).
\end{equation}
Equality holds if and only if the manifold is K\"ahler-Einstein with negative scalar curvature.
Moreover, the weighted $L^2$-norm of the weighted scalar curvature is monotone decreasing along Ricci flow, being constant if and only if the flow is Einstein.
In particular, the above estimate improves in time along the Ricci flow.
LeBrun's original estimate is a foundational result in 4-dimensional Riemannian geometry because it provides the only known obstructions to the existence of Einstein metrics \cite{L96}, other than the Hitchin-Thorpe inequality. 
The estimate (\ref{eqn: weighted L2 intro}) has also been proven via different means in \cite{FZ07}, and was used to prove existence of singularities for normalized Ricci flows \cite{FZZ08-1, I08}.

\subsection{Singularities}
It is a fundamental open problem to determine when two homeomorphic 4-manifolds are diffeomorphic.
The monopole equations are the primary tool used to \textit{distinguish} homeomorphic manifolds via the associated Seiberg-Witten invariant.
On the other hand, showing that two homeomorphic manifolds are \textit{diffeomorphic} requires very different tools. 
In three dimensions, Hamilton-Perelman's Ricci flow with surgery proved to be the key tool to answer questions of this sort.
Because of the fundamental nature of monopoles in four dimensions, Ricci flow will need to account for these objects if it is to play a role in the classification of smooth structures.

As a first step towards understanding exotic smooth structures via Ricci flow, this paper shows that singularities of the normalized Ricci flow must always exist on exotic K3 surfaces -- manifolds which are homeomorphic, but not diffeomorphic, to the K3 surface.
Existence of Ricci flow singularities on certain exotic 4-manifolds was first proven by Ishida \cite{I08}, though the manifolds he considered are topologically more complex than the K3 surface.
Indeed, the topological K3 manifold is the simplest 4-manifold after $S^4$, $\CP^2$, and $S^2\times S^2$, having the minimal second Betti number among all closed, simply-connected, spin 4-manifolds with nonzero signature \cite{Do86}.
Further, many exotic smooth structures on K3 have explicit descriptions \cite{FS98}, making them amenable to study via Ricci flow.

The existence of singularities of the normalized Ricci flow on exotic K3 surfaces is a consequence of a sharp parabolic Hitchin-Thorpe inequality (Corollary \ref{cor: parabolic HT}):
if a closed, simply-connected, spin 4-manifold $M$ admits a nonsingular solution $g(t)$ of the normalized Ricci flow, then
\begin{equation}\label{eqn: Hitchin-Thorpe intro}
    2\chi\geq 3|\sigma|,
\end{equation}
where $\chi$ denotes the Euler characteristic, and $\sigma$ the signature of $M$. 
Moreover, equality holds in (\ref{eqn: Hitchin-Thorpe intro}) if and only if $M$ is diffeomorphic to K3 and $g(t)$ converges subsequentially in the smooth Cheeger-Gromov sense to a hyperk\"ahler metric.
In particular, if equality holds, then $M$ is diffeomorphic to the standard K3, and consequently, any normalized Ricci flow on an exotic K3 surface must become singular. 
Said corollary follows by combining results of \cite{FZZ08-1} with the characterization of equality proved here (Theorem \ref{thm: parabolic Hitchin rigidity}).
This characterization of equality is a parabolic version of Hitchin's \cite{Hi74b} rigidity result for Einstein 4-manifolds, which states that the K3 surface with a hyperk\"ahler metric is the unique closed, simply-connected, Einstein 4-manifold satisfying equality in (\ref{eqn: Hitchin-Thorpe intro}).

%%%%%%%%%%%%%%%%%%%%%%%%%%%%%%%%%%
\section*{\small\bf Acknowledgements}
The author thanks W. Minicozzi for continual support, 
C. Taubes for suggesting the study of Ricci flow on spin manifolds,
and T.\ Collins, H.\ J. Hein, T.\ Ilmanen, and T.\ Ozuch for useful discussions. 
The author is a National Science Foundation Graduate Research Fellow.
% \newpage
\section{Preliminaries}

The spin bundle, and hence the Dirac operator, depends on a choice of Riemannian metric. For two choices of Riemannian metrics, the spin bundles are isomorphic, though in general not canonically so. 
However, given a 1-parameter family of Riemannian metrics, there does exist a natural identification of the spin bundles at different times, obtained via the generalized cylinder construction.
In the context of Ricci flow, this identification of the spin bundles is explained in \cite{BO2} and will be used implicitly in what follows.
This paper also employs tools from the theory of spin geometry on weighted manifolds, developed in \cite[\S 1]{BO1}. 
Because they are central to the arguments that follow, the relevant facts are reviewed below. 

\subsection{Weighted spin geometry}\label{sec: Weighted Dirac operator}

A spin structure on an oriented Riemannian manifold $(M^n,g)$ of dimension $n\geq 3$ is a two-fold covering $\tilde{P}\to P$ of the principal $\SO(n)$-bundle $P\to M$ of oriented, orthonormal frames by a principal $\Spin(n)$-bundle $\tilde{P}\to M$ restricting to the the universal covering map $\Spin(n)\to \SO(n)$ over each fiber. 
When $n=2$, the definition differs from higher dimensions because $\Un_1\cong \SO(2)\cong \Spin(2)$; in this case, a spin structure is defined as a two-fold covering $\tilde{P}\to P$ restricting to the squaring map $z\mapsto z^2$ over each fiber.
Spin structures are in general not unique; they are indexed by $H^1(M;\Z_2)$, which acts freely and transitively on the set of spin structures. 

The (complex) spin bundle $\Sigma M\to M$ is a complex vector bundle on $M$ of rank $2^{\lfloor \frac{n}{2}\rfloor}$ equipped with a Hermitian metric $\langle \cdot ,\cdot\rangle$ and a compatible connection $\nabla$. It is defined as the associated vector bundle $\Sigma M = \tilde{P} \times_{\rho_n} \Sigma_n$, where $\rho_n:\Spin(n)\to \End(\Sigma_n)$ is the standard spin representation.
A vector field on $M$ induces an endomorphism of the spin bundle via Clifford multiplication, denoted ``$\cdot$''. This endomorphism satisfies, for all vector fields $X,Y$, the Clifford algebra identity
\begin{equation}\label{eqn: Clifford algebra identity}
    X\cdot Y\cdot +\;Y\cdot X\cdot = -2\langle X,Y\rangle\id.
\end{equation}
Furthermore, Clifford multiplication by unit vectors is unitary for the spin metric, and is covariantly constant with respect to the spin connection. See \cite{BHM+} for background on spin geometry.

A weighted spin manifold is a spin manifold $(M^n,g)$ equipped with a function $f:M\to \R$ defining the weighted measure $e^{-f}dV$.
The weighted Dirac operator $D_f:\Gamma(\Sigma M)\to \Gamma(\Sigma M)$ of a weighted spin manifold is defined as 
\begin{equation}\label{eqn: weighted Dirac}
    D_f
    = e^{f/2}De^{-f/2}
    = D-\frac{1}{2}(\nabla f) \cdot \;,
\end{equation}
where $D=e_i\cdot \nabla_i$ is the standard Dirac operator, namely the composition of the spin covariant derivative $\nabla$ with Clifford multiplication. 
(Throughout this paper, 1-forms and vector fields are often identified via the metric without explicit mention.)

On closed manifolds, the weighted Dirac operator satisfies the integration by parts formula
\begin{equation}\label{eqn: weighted Laplcian IBP}
    \int_M\langle \psi,D_f\phi\rangle e^{-f}\,dV
    =\int_M \langle D_f \psi,\phi\rangle e^{-f}\,dV,
\end{equation}
and hence is self-adjoint with respect to the $L^2_f =L^2(e^{-f}\,dV)$-inner product. 
Note that, by definition, $D_f$ is unitarily equivalent to $D$ via the map $L^2\to L^2_f$ defined by $\psi\mapsto e^{f/2}\psi$.

Furthermore, the weighted Dirac operator satisfies a weighted Schr\"odinger-Lichnerowicz formula, which was observed by Perelman \cite{P02}. To state it, let
$
% \begin{equation}\label{eqn: defn of weighted Laplacian}
    \dL_f=\Delta  -\nabla_{\nabla f}
% \end{equation} 
$
be the weighted Laplacian and
\begin{equation}\label{eqn: defn of weighted scalar curvature}
    \Scal_f=\Scal+2\Delta_f f+|\nabla f|^2
\end{equation}
be Perelman's weighted scalar curvature (or P-scalar curvature).
Then the square of $D_f$ satisfies
\begin{equation}\label{eqn: weighted Schr\"odinger-Lichnerowicz}
    D_f^2=-\dL_f+\frac{1}{4}\Scal_f.
\end{equation}
Combined with the self-adjointness of $D_f$ on the weighted $L^2$ space, this implies the crucial formula
\begin{equation}\label{eqn: weighted integration by parts}
    \frac{1}{4}\int_M\Scal_f|\psi|^2\,e^{-f}dV 
    =\int_M\left( |D_f\psi|^2-|\nabla \psi|^2\right) e^{-f}dV.
\end{equation}

Additionally, the weighted (Bakry-\'Emery) Ricci curvature $\Ric_f=\Ric+\Hess_f$ is proportional to the commutator of $D_f$ and $\nabla$: for any vector field $X$ and spinor $\psi$, the following weighted Ricci identity holds
\begin{equation}\label{eqn: weighted Ricci identity}
    [D_f,\nabla_X]\psi=\frac{1}{2}\Ric_f(X)\cdot\psi.
\end{equation}
The reader is referred to \cite[\S 1]{BO1} for a comprehensive exposition of weighted spin geometry.

\subsection{Weighted harmonic spinors}
On a weighted spin manifold $(M^n,g,f)$, a weighted harmonic spinor is a spinor lying in the kernel of the weighted Dirac operator; that is, satisfying the equation
\begin{equation}
    D_f\psi =0.
\end{equation}
This equation is more natural than the usual Dirac equation on weighted manifold and is a first-order, linear, elliptic equation. 
Due to the unitary equivalence of $D$ and $D_f$, the map $\phi\mapsto e^{f/2}\phi$ is an isomorphism of the vector spaces $\ker{D}$ and $\ker{D_f}$. 

The most powerful existence theorem for harmonic spinors on closed manifolds is the Atiyah-Singer index theorem (see \cite{LM89}, for example).
The theorem applies to even-dimensional manifolds, for which the spin bundle decomposes as a direct sum $\Sigma =\Sigma^+\oplus \Sigma^-$ into the $\pm 1$-eigenspaces of the complex volume element $\omega_{\C}\in \Gamma (\End(\Sigma))$, defined locally with respect to an orthonormal frame by $\omega_{\C}=i^{\lfloor\frac{n+1}{2}\rfloor}e_1\cdots e_n\cdot$.
The weighted Dirac operator interchanges the two summands $\Sigma^{\pm}$. 
Its restriction to $\Sigma^{+}$, denoted $D_f^+$, is a zeroth-order perturbation of $D^+$, hence the indices of these two operators coincide.
In particular, it follows from the the Atiyah-Singer index theorem that the index of $D^+_f$ is equal to a topological invariant of $M$, the $\hat{A}$-genus, which is determined by the Pontryagin numbers of $M$.
In four dimensions, the index theorem takes on a particularly simple form: if $\sigma$ is the signature and $p_1$ the first Pontryagin number of $M^4$, then
\begin{equation}\label{eqn: index theorem}
    \ind(D^+_f)
    =\hat{A}
    =-\frac{1}{8}\sigma
    =-\frac{1}{24}p_1.
\end{equation}
In particular, if $\sigma\neq 0$, then there exist weighted harmonic spinors on $M$ for any choice of Riemannian metric and any choice of weight function $f$. 
This is the case, for example, on manifolds homeomorphic to connected sums of copies of the K3 surface and $S^2\times S^2$.

Although the index of $D_f$ is independent of the metric and the weight $f$, the dimension of the kernel of $D_f$ in general depends upon the metric. In particular, this dimension may change along the Ricci flow. 
The following example demonstrates this on the Berger spheres.

\begin{example}
[Ricci flow on Berger spheres]
Hitchin \cite{Hi74} computed the dimension of the space of harmonic spinors on certain manifolds with symmetries, and most notably, on the Berger spheres. 
These constitute a 1-parameter family of left-invariant Riemannian metrics on $S^3$. 
Geometrically, this family can be obtained from the round metric by shrinking along fibers of the Hopf fibration.
The standard metric on $S^3$ is very special if $S^3$ is regarded as a compact Lie group (($\SU(2)$, $\Spin(3)$, or $\mathrm{Sp}(1)$): it corresponds to the \textit{bi}-invariant metric.
If $X,Y$ are left-invariant vector fields and $g$ is a left-invariant metric, then $g_p(X_p,Y_p)=g_e(X_e,Y_e)$ for any $p\in S^3$, where $e$ denotes the identity.
Since left-invariant vector fields span the tangent space at every point $p$, a left-invariant metric is defined by a metric on the tangent space at the identity, that is, the Lie algebra.

Hence, relative to a basis $\{e_1,e_2,e_3\}$ of the Lie algebra which is orthonormal with respect to the bi-invariant metric, a left-invariant metric is diagonal and determined by its eigenvalues $\kappa_1,\kappa_2,\kappa_3>0$.
The Berger metrics correspond to left-invariant metrics which are right-invariant under $S^1\subset S^3$, when $S^3$ is considered as the total space of the Hopf fibration. This is equivalent (up to scaling), to the case $1=\kappa_2=\kappa_3$. Set $\kappa_1=\kappa^2$ and denote the corresponding metric
\begin{equation}
    g_{\kappa}= 
        \begin{bmatrix}
        \kappa^2 & & \\
        & 1 & \\
        & & 1
        \end{bmatrix}.
\end{equation}

The tangent bundle is parallelized by a basis for the Lie algebra and so the spin bundle is parallelized by the corresponding spinor basis.
Thus, relative to a left-invariant metric, the Dirac operator is a $2\times 2$ matrix of first-order, linear, left-invariant differential operators, that is, elements of the Lie algebra, and constants which are determined by $\kappa$. Using this, Hitchin \cite[\S3]{Hi74} showed that the dimension of kernel of the Dirac operator (the number of linearly independent harmonic spinors) is equal to the number of positive integer solutions $(p,q)$ to the equation
\begin{equation}\label{eqn: dim ker D on Berger}
    \kappa^2=2\sqrt{4pq\kappa^2+(p-q)^2}.
\end{equation}
In particular, if $m>0$ is an integer and $\kappa=4m$, then there are at least $2m$ linearly independent harmonic spinors.
On the other hand, Lichnerowicz's vanishing theorem implies that there exist no harmonic spinors if the scalar curvature is positive. Straightforward computation \cite[\S3.3]{Hi74} shows that the scalar curvature of the metric $g_{\kappa}$ is determined by $\kappa$ as follows
\begin{equation}\label{eqn: Scal of Berger}
    \Scal = 2(4-\kappa^2).
\end{equation}

Consider the volume-normalized Ricci flow $g(t)$ starting at the Berger metric $g_{\kappa}$, where $\kappa\geq 16$, so that the metric admits harmonic spinors and $\Scal<0$. 
Because the Ricci flow preserves the isometry group, up to scaling, the metric remains in the form $g_{\kappa(t)}$, and $\kappa(t)$ evolves according to an ODE derived from the Ricci flow equation.
Ricci flow tends to increase the minimum of scalar curvature, which, due to (\ref{eqn: Scal of Berger}), means $\kappa(t)$ is decreasing.
In fact, ODE analysis \cite[\S3(b)]{IJ92} reveals that the volume-normalized Ricci flow starting at $g_{\kappa}$ exists for all time and converges to the round metric exponentially in time as $t\to \infty$.
In particular, there exists a time $t_0>0$ such that the Berger metric $g_{\kappa(t_0)}$ has positive scalar curvature, and hence no harmonic spinors.
This shows that the dimension of the kernel of the Dirac operator can decrease in time along a volume-normalized Ricci flow.
Because the volume-normalized Ricci flow differs from the regular Ricci flow $\partial_tg=-2\Ric$ only by a rescaling of space and reparameterization of time, the same is true for regular Ricci flow.

\end{example}
%%%%%%%%%%%%%%%%%%%%%%%%%%%%%%%%%%
\section{Entropy and harmonic spinors}

This section shows that Perelman's $\lambda$-entropy can be expressed in terms of the energy of a harmonic spinor. Furthermore, the relationship between Perelman's $\lambda$-entropy and the first eigenvalue of the square of the Dirac operator is investigated in the context of Ricci flow.

Let $(M^n,g)$ be a closed Riemannian manifold.
Perelman's $\lambda$-entropy is defined as the first eigenvalue of the operator $-4\Delta +\Scal$, or equivalently, as the minimum of the weighted Hilbert-Einstein functional \cite{P02}:
\begin{equation}
    \lambda(g)
    =\inf_{\int u^2=1}\;\int_M\left(4|\nabla u|^2+\Scal u^2\right)dV
    =\inf_{\int e^{-f}=1}\;\int_M\Scal_fe^{-f} \,dV.
\end{equation}
A function $f$ is the minimizer of the $\lambda$-entropy if and only if $\Scal_f$ is constant, with $\Scal_f=\lambda(g)$.
On manifolds admitting a harmonic spinor, the following theorem provides a formula for the $\lambda$-entropy in terms of a weighted harmonic spinor.
This result should be compared to the case of Ricci flows on asymptotically Euclidean spin manifolds, where the analogue of the $\lambda$-entropy can always be expressed as the weighted Dirichlet energy of a suitable weighted Witten spinor \cite{BO1,BO2}.

\begin{theorem}
\label{thm: lambda-entropy in terms of harmonic spinor}
Let $f$ be the minimizer of Perelman's $\lambda$-entropy on a closed manifold $(M^n,g)$ which admits a harmonic spinor $\phi$. Then the spinor $\psi=e^{f/2}\phi$ satisfies
\begin{equation}\label{eqn: formula for lambda in terms of harmonic spinor}
    \lambda(g)
    =-4\frac{\|\nabla \psi\|_{L^2_f}^2}{\|\psi\|_{L^2_f}^2}.
\end{equation}
\end{theorem}

\begin{proof}
Suppose $D\phi=0$ for some nontrivial spinor $\phi$. 
Since $D_f=e^{f/2}De^{-f/2}$ is a unitary equivalence, if $\psi=e^{f/2}\phi$, then $D_f\psi=0$ and  $\|\psi\|_{L^2_f}=\|\phi\|_{L^2}$.
Further, since $\Scal_f=\lambda(g)$ is constant, Equation (\ref{eqn: weighted integration by parts}) implies
\begin{equation}
    \lambda(g)
    =\frac{\int_M\Scal_f |\psi|^2\,e^{-f}dV}{\int_M |\psi|^2\,e^{-f}dV}
    =\frac{\int_M 4\left(|D_f\psi|^2-|\nabla \psi|^2\right)e^{-f}dV}{\int_M |\psi|^2\,e^{-f}dV}
    =-\frac{4\int_M|\nabla \psi|^2\,e^{-f}dV}{\int_M |\psi|^2\,e^{-f}dV}.
\end{equation}
\end{proof}

This theorem has important consequences along a Ricci flow.
Let $(M^n,g(t))$ be a Ricci flow evolving by $\partial_tg=-2\Ric$ on a closed manifold admitting a harmonic spinor at each time, for example, if $\ind(D)\neq 0$. 
Let $f$ be the minimizer of Perelman's $\lambda$-entropy at each time along a Ricci flow. 
Then Perelman's monotonicity formula states that
\begin{equation}\label{eqn: lambda monotonicity}
    \lambda'(g(t)) = 2\int_M|\Ric_f|^2\,e^{-f}dV.
\end{equation}
This formula can be recast in terms of a harmonic spinor as follows.
If $D_f\psi=0$, then the weighted Ricci identity (\ref{eqn: weighted Ricci identity}) implies
$
    \frac{1}{4}|\Ric_f|^2|\psi|^2
    =|D_f\nabla \psi|^2.
$
Combined with (\ref{eqn: lambda monotonicity}), this implies the following monotonicity formula
\begin{equation}\label{eqn: monotonicity rewritten}
    \frac{d}{dt}\left(\frac{\|\nabla \psi\|_{L^2_f}^2}{\|\psi\|_{L^2_f}^2}\right)
    =-2\int_M\frac{|D_f\nabla \psi|^2}{|\psi|^2}\,e^{-f}dV.
\end{equation}
Furthermore, Perelman's differential inequality %\cite{P02}
$    
    \lambda'(g(t)) \geq \frac{2}{n}\lambda^2(g(t)),
$
can be recast as a differential inequality for the energy of a harmonic spinor along the flow:
\begin{equation}\label{eqn: monotonicity inequality}
    \frac{d}{dt}\left(\frac{\|\nabla \psi\|_{L^2_f}^2}{\|\psi\|_{L^2_f}^2}\right)
    \leq -\frac{8}{n}\left(\frac{\|\nabla \psi\|_{L^2_f}^2}{\|\psi\|_{L^2_f}^2}\right)^2,
\end{equation}
with equality if and only if the metric is Einstein. 

To deal with the expanding case, define Perelman's normalized $\lambda$-entropy as
\begin{equation}\label{eqn: normalized lambda}
    \bar{\lambda}(g)=\lambda(g)\Vol_g(M^n)^{\frac{2}{n}}.
\end{equation}
This quantity is invariant under scaling of the metric. Perelman's \cite{P02} inequality $\bar{\lambda}'\geq 0$ then implies the following scale-invariant monotonicity formula
\begin{equation}\label{eqn: scale inv monotonicity rewritten}
    \frac{d}{dt}\left(\frac{\|\nabla \psi\|_{L^2_f}^2}{\|\psi\|_{L^2_f}^2}\Vol(M^n)^{\frac{2}{n}}\right)
    \leq 0,
\end{equation}
with equality if and only if the manifold is Einstein.
This inequality is particularly interesting in four dimensions; see Corollary \ref{cor: weighted LeBrun estimate}.

%%%%%%%%%%%%%%%%%%%%%%%%%%%%%%%%%%
\subsection{Eigenvalue evolution}\label{sec: Negative Einstein manifolds}

Perelman's $\lambda$-entropy is equal to the first eigenvalue of the operator $L=-4\Delta +\Scal$ acting on functions, and Ricci flow is the gradient flow of this quantity.
Due to the Schr\"odinger-Lichnerowicz formula, the operators $4D^2=-4\Delta+\Scal$ and Perelman's operator $L$ appear to be closely related.
Despite their formal similarity, the evolution of the smallest eigenvalues of said operators under Ricci flow is very different: $\lambda$ is monotone increasing, whereas the first eigenvalue $\lambda_1$ of the square of the Dirac operator may decrease, as is shown below.
What allows for this differing behavior is the fact that $\lambda_1$ is always non-negative, whereas $\lambda$ may be negative.

Let $(M^n,g,f)$ be a closed gradient Ricci soliton admitting a spin structure. 
This means that the metric satisfies the equation
\begin{equation}
    \Ric+\Hess_f = \frac{\epsilon}{2}g,
\end{equation}
for some $\epsilon = -1,0,1$. 
The soliton is called expanding if $\epsilon=-1$, steady if $\epsilon=0$, and shrinking if $\epsilon=1$.
Corresponding to this gradient Ricci soliton, there exists a Ricci soliton flow $g(t)$, defined for all $t\in \R$ such that $\tau(t)>0$, where $\tau(t)=-\epsilon t$ in the expanding and shrinking case, and $\tau(t)=1$ in the steady case. 
The flow evolves only by diffeomorphisms and scalings, and satisfies $g(-\epsilon)=g$.

In particular, scalar geometric quantities such as $\lambda_1$ evolve only by scaling along the flow:
under a Ricci soliton flow, it follows from the fact that $g(t)$ is conformal to $g$ (up to a diffeomorphism) with conformal factor $\tau(t)$, that
\begin{equation}
    \lambda_1(g(t))=\frac{\lambda_1(g)}{\tau(t)}.
\end{equation}
Therefore, if $\lambda_1$ is nonzero (meaning no harmonic spinors exist), then this eigenvalue is decreasing on expanding soliton flows, constant on steady soliton flows, and increasing on shrinking soliton flows.
For example, the round sphere $S^n$ with its constant curvature metric induces a shrinking Ricci soliton flow admitting no harmonic spinors; hence $\lambda_1$ is increasing along such a flow.
% (This may also be deduced from the evolution ODE for the scalar curvature on $S^n$, and the fact that $\lambda_1=\frac{n}{4(n-1)}\Scal$ on $S^n$.)
Below examples are given of expanding soliton flows for which $\lambda_1$ is strictly decreasing. These flows are induced by hyperbolic surfaces equipped with spin structures admitting no harmonic spinors.

\begin{example}
[Hyperbolic surfaces without harmonic spinors]\label{ex: negative Einstein manifold with no harmonic spinors}
A closed, oriented surface $M^2$ of genus $\mathfrak{g}$ has $|H^1(M;\Z_2)|=2^{2\mathfrak{g}}$ inequivalent spin structures.
Hitchin \cite[Prop. 2.3]{Hi74} showed that if $\mathfrak{g}\leq 2$, the dimension of the kernel of the Dirac operator is independent of the Riemannian metric, and depends only on the spin structure.
Moreover, he showed that if $\mathfrak{g}=2$, then there are 6 spin structures admitting a harmonic spinor, while the remaining 10 admit none. 
In particular, if $M^2$ is a two-holed torus ($\mathfrak{g}=2$), there exists a spin structure on $M^2$ admitting no harmonic spinors for any choice of metric. Equipping $M^2$ with a hyperbolic metric $g$, it follows that $\lambda_1(g)>0$, and along the associated expanding soliton flow $g(t)=tg$, the eigenvalue $\lambda_1(g(t))$ is decreasing. 
\end{example}

\section{Entropy and monopoles}
\label{sec: entropy and monopoles}

A spin-c structure $\mathfrak{s}$ on an oriented Riemannian 4-manifold $(M^4,g)$ determines, up to isomorphism, a complex line bundle $L\to M$ such that $c_1(L)\equiv w_2(TM) \mod 2$ and two rank-2 Hermitian vector bundles $W^{\pm}\to M$ such that $\det(W^{\pm})=L$ and $T^*M\otimes \C\cong \Hom(W^+,W^-)$.
The latter isomorphism corresponds to Clifford multiplication, and satisfies the Clifford algebra identity (\ref{eqn: Clifford algebra identity}).

On any contractible open set in $M$, there exists a canonical isomorphism $W^{\pm} \cong \Sigma^{\pm}\otimes L^{\frac{1}{2}}$, where $\Sigma^{\pm}$ are the (locally defined) positive and negative spinor bundles of $g$, and $L^{\frac{1}{2}}$ is a complex line bundle whose square is $L$.
Each unitary connection $A$ on $L$ thus determines a unitary connection $\nabla^A:\Gamma(W^+)\to \Gamma(W^+\otimes T^*M)$ on $W^+$, and composing this with the isomorphism $T^*M\otimes \C \cong \Hom(W^+,W^-)$ induces the spin-c Dirac operator $D_A:\Gamma(W^+)\to \Gamma(W^-)$. 
On certain spin-c 4-manifolds, including K\"ahler manifolds, there exists a canonical choice of connection $A$, but in general, no canonical choice exists.
The Seiberg-Witten monopole equations (\ref{eqn: monopole equation}) determine a special class of connections.

\subsection{Weighted monopoles}

Let $(M^4,g,f,\mathfrak{s})$ be a closed, weighted, spin-c 4-manifold. 
The spin-c Dirac operator defined by a $U(1)$-connection $A$ on the associated complex line bundle $L\to M$ is locally a twisted Dirac operator. Therefore, Proposition \ref{prop: Weighted Schrodinger-Lichnerowicz for twisted Dirac operators} implies the following weighted Schr\"odinger-Lichnerowicz formula
\begin{equation}\label{eqn: weighted spin-c Schr\"odinger-Lichnerowicz formula}
D_{A,f}^2=-\Delta_{A,f}+\frac{1}{4}\Scal_f +\frac{1}{2}F_A\cdot,
\end{equation}
where $D_{A,f}=e^{f/2}D_Ae^{-f/2}$ is the weighted spin-c Dirac operator.
The monopole equations for a pair $(A,\phi)$ are the coupled system given by
\begin{align}\label{eqn: monopole equation}
        F_A^+\cdot &= (\phi \otimes \phi^*)^{\circ} \\
        D_A\phi&=0. \nonumber
\end{align}
The first equation takes place in the space of traceless endomorphisms of $W^+$, with $(\phi \otimes \phi^*)^{\circ}=\phi\otimes \phi^*-\frac{1}{2}|\phi|^2\id$ denoting the traceless part of the endomorphism $\phi\otimes \phi^*\in \Gamma(\End(W^+))$.

If $(A,\phi)$ is a solution of the monopole equations, then the spinor $\psi=e^{f/2}\phi$ together with the connection $A$ satisfies the \textit{weighted monopole equations}
\begin{align}\label{eqn: weighted monopole equation}
        F_A^+\cdot &= e^{-f}(\psi \otimes \psi^*)^{\circ} \\
        D_{A,f}\psi&=0. \nonumber
\end{align}
Note that the exponential term in the first equation stems from the nonlinearity of the monopole equations.
The weighted Schr\"odinger-Lichnerowicz formula (\ref{eqn: weighted spin-c Schr\"odinger-Lichnerowicz formula}) combined with (\ref{eqn: weighted monopole equation}) then implies
\begin{align}\label{eqn: weighted square of monopole}
    0
    =D_{A,f}^2\psi
    &=-\Delta_{A,f}\psi+\frac{1}{4}\Scal_f \psi+\frac{1}{2}F_A^+\cdot\psi 
    \\\nonumber
    &=-\Delta_{A,f}\psi+\frac{1}{4}\Scal_f \psi+\frac{1}{4}e^{-f}|\psi|^2\psi.
\end{align}
Upon taking the inner product with $\psi$, it follows that
\begin{align}\label{eqn: intermediate weighted SW}
    0
    &=-\Rea \langle \Delta_{A,f}\psi,\psi\rangle+\frac{1}{4}\Scal_f|\psi|^2 +\frac{1}{4}e^{-f}|\psi|^4
    \\\nonumber
    &=-\frac{1}{2}\Delta_f|\psi|^2 +|\nabla^A\psi|^2+\frac{1}{4}\Scal_f|\psi|^2 +\frac{1}{4}e^{-f}|\psi|^4.
\end{align}
This equation implies a generalization of the well-known $C^0$-estimate for monopoles, which is crucial for the compactness theory of the monopole equations and applications to scalar curvature geometry.

\begin{lemma}
    [Weighted $C^0$-estimate]\label{lem: weighted C0 estimate}
    Let $(A,\psi)$ be a solution of the weighted monopole equations (\ref{eqn: weighted monopole equation}) on a closed, weighted 4-manifold $(M^4,g,f)$. Then
    \begin{equation}
        \max_M\;|\psi|^2 \leq -\min_M \Scal_f\,e^f.
    \end{equation}
\end{lemma}

\begin{proof}
    If $\psi$ is not identically zero, then (\ref{eqn: intermediate weighted SW}) implies
    \begin{equation}
        |\psi|^2
        =2\frac{\Delta_f |\psi|^2}{|\psi|^2}e^f -4\frac{|\nabla^A\psi|^2}{|\psi|^2}e^f-\Scal_f\,e^f.
    \end{equation}
    At a maximum of $|\psi|^2$, $\nabla |\psi|^2=0$ and $\Delta |\psi|^2\leq 0$, so $\Delta_f|\psi|^2\leq 0$. The second term on the right hand side is also non-positive, so the claim follows.
\end{proof}

A spin-c structure $\mathfrak{s}$ on $M$ is called a \textit{monopole class} if the monopole equations (\ref{eqn: monopole equation}) have a solution for every Riemannian metric on $M$.
(Spin-c structures are in bijection with $H^2(M;\Z)$ when the latter is torsion-free, explaining the term ``class''.)
In particular, if $\mathfrak{s}$ is a monopole class, then there exist weighted monopoles on $M$ for any choice of Riemannian metric and any choice of weight function $f$.
A solution $(A,\psi)$ of the weighted monopole equations (\ref{eqn: weighted monopole equation}) is called \textit{reducible} if $\psi\equiv 0$; otherwise, the solution is called \textit{irreducible}.
The monopole equations combined with the formula
\begin{equation}\label{eqn: Chern-Weil}
    \int_M\left(|F_A^+|^2-|F_A^-|^2\right) dV 
        = 4\pi^2 c_1^2(\mathfrak{s}),
\end{equation}
imply that all solutions of the monopole equations are irreducible if $c_1^2(\mathfrak{s})>0$;
in particular, all monopole classes satisfying this condition admit irreducible monopoles.

The condition that $\mathfrak{s}$ is a monopole class may be viewed as an analogue of the condition that the index of the Dirac operator on a spin manifold is nontrivial -- both are topological conditions guaranteeing the existence of solutions.
Taubes showed that the canonical spin-c structure of a symplectic 4-manifold with $b_2^+\geq 2$ is a monopole class \cite{T00}; this result applies, for example, to exotic K3 surfaces (manifolds homeomorphic, but not diffeomorphic, to the K3 surface) which are constructed via knot surgery.
This important example is discussed in Section \ref{sec: immortal}.

The main theorem of this section follows from the integral version of (\ref{eqn: intermediate weighted SW}). It reads
\begin{equation}\label{eqn: intermediate weighted SW; integral}
    0=\int_M\left(|\nabla^A\psi|^2+\frac{1}{4}\Scal_f|\psi|^2+\frac{1}{4}e^{-f}|\psi|^4\right)\, e^{-f}dV.
\end{equation}
From this formula, it follows immediately that if $\Scal_f>0$, 
then all solutions of the monopole equations are reducible. 
In particular, the associated Seiberg-Witten invariant, which counts irreducible solutions, vanishes. 
The unweighted version of the above integral equation is the cornerstone of applications of monopole theory to the study of 4-dimensional Riemannian geometry \cite{L96}.
The weighted version derived here implies the following theorem.

\begin{theorem}
[$\lambda$-entropy in terms of monopole]\label{thm: lambda-entropy in terms of monopole}
Let $f$ be the minimizer of Perelman's $\lambda$-entropy on a closed manifold $(M^4,g)$ admitting an irreducible monopole $(A,\phi)$. Then $\psi=e^{f/2}\phi$ satisfies
\begin{equation}
    \lambda(g)
    =-\frac{\int_M\left(4|\nabla^A \psi|^2+e^{-f}|\psi|^4\right)\,e^{-f}dV}{\int_M|\psi|^2\,e^{-f}dV}.
\end{equation}
\end{theorem}

\begin{proof}
With the given choice of $f$, $\Scal_f$ is constant with $\Scal_f=\lambda(g)$. Since the monopole is irreducible, the proof is completed by rearranging (\ref{eqn: intermediate weighted SW; integral}) to solve for $\Scal_f$.
\end{proof}

\begin{remark}
[Monotonicity formulas]
The theorem combined with Perelman's monotonicity formula for $\lambda$ along the Ricci flow immediately implies monotonicity formulas for the weighted energy of a monopole along the Ricci flow similar to (\ref{eqn: monotonicity rewritten}), (\ref{eqn: monotonicity inequality}) and (\ref{eqn: scale inv monotonicity rewritten}). The proofs follow the same reasoning as for the case of regular harmonic spinors given in the previous section. 
\end{remark}

Theorem \ref{thm: lambda-entropy in terms of monopole} also implies a scale-invariant, weighted $L^2$ scalar curvature estimate.
The estimate has been stated in a different form and derived by different means in \cite{FZ07}; 
expressed in terms of the weighted scalar curvature $\Scal_f$, the estimate appears as the natural weighted version of LeBrun's 
$L^2$ scalar curvature estimate, which is central to applications of monopole theory to 4-dimensional Riemannian geometry \cite{L95}.
Further, Theorem \ref{thm: lambda-entropy in terms of monopole} allows for a simpler proof.

\begin{corollary}
    [Weighted $L^2$ curvature estimate]\label{cor: weighted LeBrun estimate}
    Let $\mathfrak{s}$ be a monopole class on a closed 4-manifold $(M^4,g)$ such that $c_1^2(\mathfrak{s})>0$, and let $f$ be the minimizer of Perelman's $\lambda$-entropy with $\int_Me^{-f}dV=\Vol_g(M)$.
    Then the following scale-invariant inequality holds
    \begin{equation}\label{eqn: weighted L2 scalar estimate}
            \int_M\Scal_f^2\,e^{-f}dV
            \geq 32\pi^2 c_1^2(\mathfrak{s}),
        \end{equation}
    with equality if and only if $g$ is K\"ahler-Einstein with negative scalar curvature.
    Moreover, along Ricci flow, the left-hand-side is monotone decreasing and constant if and only if the manifold is Einstein.
\end{corollary}

\begin{proof}
Let $(A,\phi)$ be an irreducible monopole, so that $(A,\psi)$ is an irreducible weighted monopole, where $\psi=e^{f/2}\phi$.
Theorem \ref{thm: lambda-entropy in terms of monopole}, implies that 
\begin{align}
    \int_M\Scal_f^2\,e^{-f}dV
    &=\Vol_g(M)\lambda^2(g)
    \\ \nonumber 
    &= \Vol_g(M)\left(\frac{\int_M\left(4|\nabla^A \psi|^2+e^{-f}|\psi|^4\right)\,e^{-f}dV}{\int_M|\psi|^2\,e^{-f}dV}\right)^2
    \\ \nonumber 
    &\geq \Vol_g(M)\left(\frac{\int_M e^{-2f}|\psi|^4 \,dV}{\int_M|\psi|^2\,e^{-f}dV}\right)^2.
\end{align}
The fact that $e^{-2f}|\psi|^4=|\phi|^4=|F_A^+\cdot|^2$ and the Cauchy-Schwarz inequality then imply
\begin{align}
    \int_M\Scal_f^2\,e^{-f}dV
    \geq \Vol_g(M) \left( \frac{\int_M|\phi|^4\,dV}{\int_M|\phi|^2\,dV} \right)^2
    % \\ \nonumber 
    \geq  \int_M|\phi|^4\,dV
    % \\ \nonumber 
    = 8\int_M|F_A^+|^2\,dV.
\end{align}
% Since $|F_A^+\cdot|^2=8|F_A^+|^2$, it follows from (\ref{eqn: Chern-Weil}) that
Combined with (\ref{eqn: Chern-Weil}), it then follows that
\begin{align}
    \int_M\Scal_f^2\,e^{-f}dV
    \geq  8\int_M|F_A^+|^2\,dV
    =  32\pi^2c_1^2+8\int_M|F_A^-|^2\,dV
    \geq  32\pi^2c_1^2.
\end{align}

Equality occurs only if $\nabla^A\psi=0$. Since the solution is irreducible, $\psi$ is a nontrivial parallel spinor.
This implies the manifold is K\"ahler, with K\"ahler form $\omega=\frac{\sqrt{2}}{|F_A^+|}F_A^+$, since the self-dual 2-form $F_A^+$ is non-degenerate and parallel.
To show that the manifold must also be Einstein, note that with this choice of $f$, $\int_M\Scal_f^2\,e^{-f}dV=\bar{\lambda}^2(g)$, where $\bar{\lambda}$ is the normalized $\lambda$-entropy (\ref{eqn: normalized lambda}). 
Hence, if equality holds, then $\bar{\lambda}(g)$ is a global maximum of the normalized $\lambda$-entropy, and is thus constant along any Ricci flow.
Perelman's \cite{P02} scale-invariant monotonicity implies $\bar{\lambda}' \geq 0$ along a Ricci flow, with equality if and only if the flow is Einstein with negative scalar curvature. (Ricci-flat manifolds are excluded by the assumption $c_1^2>0$.) 
This proves the ``only if'' case of equality, and also that the left-hand-side of (\ref{eqn: weighted L2 scalar estimate}) is improving along a Ricci flow, being constant only in the Einstein case.

Finally, it is straightforward to check that any K\"ahler-Einstein manifold with negative scalar curvature and equipped with the canonical spin-c structure saturates the inequality in (\ref{eqn: weighted L2 scalar estimate}), since such manifolds admit parallel spin-c spinors and the minimizer $f$ of $\lambda$ is a constant function.
\end{proof}
\section{Singularities}\label{sec: immortal}

This section proves a sharp parabolic Hitchin-Thorpe inequality, and consequently, that the normalized Ricci flow on any \textit{exotic} K3 surface -- a manifold which is homeomorphic, but not diffeomorphic, to K3 -- must become singular. 
To place the result into context, recall that Thorpe \cite{Th69} showed that if a closed, simply-connected 4-manifold admits an Einstein metric, then the Euler characteristic $\chi$ and signature $\sigma$ of said manifold must satisfy the inequality
\begin{equation}\label{eqn: Hitchin-Thorpe}
    2\chi\geq 3|\sigma|.
\end{equation}
Hitchin \cite{Hi74b} later rediscovered the inequality and characterized the case of equality: equality holds in (\ref{eqn: Hitchin-Thorpe}) if and only if the manifold is isometric to the K3 surface equipped with a hyperk\"ahler metric. 
In applications to the geometry and topology of 4-manifolds, the characterization of equality is of utmost importance, since it implies, for example, that exotic K3s cannot admit Einstein metrics.

Because Einstein metrics are static solutions of the normalized Ricci flow, the normalized Ricci flow is the correct tool for a parabolic generalization of the Hitchin-Thorpe inequality.
A family of metrics $(g(t))_{t\in I}$ on a closed $n$-manifold $M$ evolves by the \textit{normalized Ricci flow} if
\begin{equation}
    \partial_tg=-2\Ric +\frac{2}{n}\bar{\Scal}g,
\end{equation}
where $\avgScal=\fint \Scal \,dV$ denotes the average scalar curvature.
The flow is called normalized because it preserves the volume of the manifold in time and is obtained from the standard Ricci flow $\partial_tg=-2\Ric$ by rescaling the metric at each time, and by reparameterizing time. 

Like the standard Ricci flow, the normalized Ricci flow is a parabolic equation for the metric, so solutions always exist for short time given any smooth initial metric.
However, the flow is nonlinear, so singularities may occur.
Hamilton \cite{Ha82} showed that the flow becomes singular if and only if the curvature blows up at some (possibly infinite) time. 
A \textit{nonsingular} normalized Ricci flow is thus defined as one which exists for all time with uniformly bounded curvature
\begin{equation*}
    \sup_{M\times [0,\infty)} |\Rm|(x,t)<\infty.
\end{equation*}
Singularities of the flow are indicators of nontrivial topology and can aid in providing a geometric decomposition of the manifold into ``standard'' pieces. This is demonstrated in three dimensions by Hamilton-Perelman's theory of Ricci flow with surgery, as well as in higher dimensions under suitable curvature conditions, including positive isotropic curvature.
For applications of Ricci flow to the study of 4-dimensional topology, it is therefore important to know when singularities exist.

The following theorem characterizes equality in the parabolic Hitchin-Thorpe inequality (Corollary \ref{cor: parabolic HT}).

\begin{theorem}\label{thm: parabolic Hitchin rigidity}
    If a closed, spin 4-manifold $M$ satisfies $2\chi = 3|\sigma|>0$ and admits a nonsingular solution $g(t)$ of the normalized Ricci flow, then $g(t)$ converges subsequentially in the smooth Cheeger-Gromov sense to a hyperk\"ahler metric on a finite quotient of K3. 
    % In particular, $\tilde{M}$ is diffeomorphic to K3.
\end{theorem}

\begin{proof}
    Let $g(t)$, for $t\geq 0$ be a nonsingular, normalized Ricci flow on $M$. 
    Assume that $\Vol(M,g(t))=1$, which may be achieved by rescaling the metric.
    The flow $g(t)$ is nonsingular, hence there exists a uniform curvature bound. Shi's derivative estimates then imply bounds on all derivatives 
    \begin{equation}
        \sup_{M\times [0,\infty)} |\nabla^k\Rm|<\infty.
    \end{equation}
    
    In particular, the flow cannot collapse: if the solution did collapse, it would do so with bounded curvature, and therefore admit an $F$-structure of positive rank \cite{CG90}. 
    The existence of such a structure implies the manifold has trivial Euler characteristic, which contradicts the assumption that $\chi(M)>0$.
    Thus, the flow has a uniform injectivity radius lower bound
    \begin{equation}
        \inf_{t\geq 0} \mathrm{inj}_{g(t)}(M)>0.
    \end{equation}

    Due to the injectivity radius and curvature bounds, Cheeger-Gromov's compactness theorem (see \cite[Ch.\ 6]{CLN}, for example) implies that for any sequence $(p_k,t_k)\in M\times [0,\infty)$ with $t_k\to \infty$, there exists a subsequence, denoted by the same index, and a complete, pointed Riemannian manifold $(M_{\infty},g_{\infty},p)$ such that
    \begin{equation}
        (M,g(t_k),p_k)\to (M_{\infty},g_{\infty},p)
    \end{equation}
    in the smooth Cheeger-Gromov sense as $k\to \infty$. By the definition of Cheeger-Gromov convergence, the limit manifold has the following properties
    \begin{equation}
        \sup_{M\times [0,\infty)} |\nabla^k\Rm|<\infty,
        \qquad \qquad
        \mathrm{inj}_{g_{\infty}}(M_{\infty})>0, 
        \qquad \qquad 
        \Vol_{g_{\infty}}(M_{\infty})\leq 1.
    \end{equation}
    The remainder of the proof consists of showing that $M_{\infty}$ is compact and hyperk\"ahler. This is achieved by showing that the the harmonic spinors converge to a parallel spinor on the limit.

    Since $M$ is spin and $\sigma\neq 0$, the index theorem (\ref{eqn: index theorem}) implies that the manifold $(M,g(t_k))$ admits a harmonic spinor $\psi_k$ at every time $t_k$. 
    Lichnerowicz' vanishing theorem therefore implies that $\minScal(t)\leq 0$ for all $t$.
    Furthermore, because $M$ satisfies equality in the Hitchin-Thorpe inequality, it follows that $\minScal(t)\to 0$ as $t\to \infty$. Indeed, if this was not the case, then $\minScal(t)\leq -c<0$ for some $c>0$ independent of $t$, and the proof of \cite[Lem.\ 3.2]{FZZ08-1} would imply that 
    \begin{equation}
        2\chi-3|\sigma|\geq \frac{c^2}{96\pi^2}>0,
    \end{equation}
    which violates the assumption that $2\chi=3|\sigma|$.
    Therefore, assuming that the harmonic spinors $\psi_k$ have unit $L^2(g(t_k))$-norm, it follows from the Schr\"odinger-Lichnerowicz formula that
    \begin{equation}\label{eqn: gradient estimate}
        \int_M|\nabla \psi_k|^2\, dV_{g(t_k)} \leq -\frac{1}{4} \minScal(t_k)  \to 0.
    \end{equation}

    Since $\lambda(g(t))\leq 0$ by Theorem \ref{thm: lambda-entropy in terms of harmonic spinor}, it follows that the volume of the unnormalized Ricci flow determined by $g(t)$ cannot go to zero in finite time, and using this combined with the Chern-Gauss-Bonnet theorem, it follows \cite{ZZ12} that, subsequentially,
    \begin{equation}
        \lim_{k\to \infty}\int_M|\trRic|^2\,dV_{g(t_k)}\to 0.
    \end{equation}
    In particular, the limit manifold $(M_{\infty},g_{\infty})$ is Ricci-flat. 
    Since a complete, noncompact manifold with nonnegative Ricci curvature has at least linear volume growth \cite{Y76}, it follows that $M_{\infty}$ must be compact because Cheeger-Gromov convergence implies $\Vol_{g_{\infty}}(M_{\infty})\leq 1$.
    
    In particular, the definition of smooth Cheeger-Gromov convergence implies that for large $k$, there exist diffeomorphisms $\Phi_k:M_{\infty}\to M$ such that 
    \begin{equation}
        \Phi_k^*g(t_k)\xrightarrow{C^{\infty}} g_{\infty}.
    \end{equation}
    Hence, $M_{\infty}$ is diffeomorphic to $M$. 
    To show that $g_{\infty}$ is hyperk\"ahler, it follows from a standard Sobolev space argument using (\ref{eqn: gradient estimate}) that the harmonic spinors $\psi_k$ converge to a parallel spinor of unit $L^2(g_{\infty})$-norm.
    Alternatively, since $g_{\infty}$ is a Ricci-flat metric on a closed 4-manifold which satisfies $2\chi=3|\sigma|$, Hitchin's classification \cite{Hi74b} implies that $(M,g_{\infty})$ is isometric to a hyperk\"ahler metric on a quotient of the flat torus or the K3 surface.
    Since a quotient of a flat torus has trivial Euler characteristic, it follows from the assumption $2\chi=3|\sigma|>0$ that the universal cover $\tilde{M}$ is isometric to a hyperk\"ahler metric on K3. In particular $\tilde{M}$ is diffeomorphic to the standard K3.

\end{proof}

Combining the previous theorem with the results of \cite{FZZ08-1, ZZ12} implies the following sharp parabolic Hitchin-Thorpe inequality.
The characterization of equality is necessary for proving that the normalized Ricci flow on exotic K3s must become singular (Corollary \ref{cor: K3 singularities}).

\begin{corollary}
    [Parabolic Hitchin-Thorpe inequality]\label{cor: parabolic HT}
    If a closed, simply-connected, spin 4-manifold $M$ admits a nonsingular solution $g(t)$ of the normalized Ricci flow, then
    \begin{equation}\label{eqn: HT cor}
        2\chi\geq 3|\sigma|,
    \end{equation}
    with equality if and only if $M$ is diffeomorphic to K3 and $g(t)$ converges subsequentially in the smooth Cheeger-Gromov sense to a hyperk\"ahler metric.
\end{corollary}

\begin{proof}
    Since $M$ admits a nonsingular solution to the normalized Ricci flow, \cite[Thm.\ 1.2]{FZZ08-1} implies that (\ref{eqn: HT cor}) holds unless $M$ admits a positive rank $F$-structure, or a shrinking Ricci soliton structure.
    The existence of a positive rank $F$-structure can be excluded on topological grounds since the existence of such a structure implies that $\chi=0$, whereas simply-connectedness implies 
    \begin{equation}\label{eqn: chi at least 2}
        \chi\geq 2.
    \end{equation}
    
    Hence, the only remaining case to consider is the one in which $M$ admits a shrinking Ricci soliton structure. 
    It is well-known that the Riemannian metric of such a structure has positive scalar curvature.
    In particular, the Dirac index of $M$ is trivial, and hence, by the index theorem (\ref{eqn: index theorem}), it follows that $\sigma=0$.
    Combined with (\ref{eqn: chi at least 2}), the inequality (\ref{eqn: HT cor}) therefore holds.
    This proves the first part of the Corollary.
    The equality case follows immediately from Theorem \ref{thm: parabolic Hitchin rigidity}; indeed, the hypothesis $2\chi=3|\sigma|>0$ is satisfied since $M$ is assumed to be simply-connected.
\end{proof}

\begin{remark}
    The hypothesis that $M$ is simply-connected in Corollary \ref{cor: parabolic HT} is necessary; indeed, without this assumption, the equality statement would have to be weakened to account for normalized Ricci flows converging to the flat metric on the 4-torus or quotients thereof. Such a statement would generalize the classical Hitchin-Thorpe inequality for Einstein 4-manifolds with nontrivial fundamental group \cite{Th69, Hi74b}.
\end{remark}

The preceding corollary is useful in proving that singularities of the normalized Ricci flow must exist on many spin 4-manifolds, particularly on those which satisfy \textit{equality} in the Hitchin-Thorpe inequality.
Freedman and Donaldson's classification theorems (see \cite[Ch.\ 5]{Sc05}, for example) imply that the intersection form of any closed, smooth, simply-connected, spin 4-manifold $M$ is given by 
\begin{equation}
    Q_{M}= -2p E_8 \oplus q 
    \begin{bmatrix}
        & 1 \\ 1 &
    \end{bmatrix},
\end{equation}
for some $p,q\in \Z$ with $q\geq 0$, where $E_8$ denotes the $E_8$-matrix.
After reversing orientations, one may assume that $p\geq 0$.
If $q\geq 3p$, then the form $Q_M$ is realized as the intersection form of $pK3\#(q-3p)S^2\times S^2$.
The $11/8$-conjecture states that the condition $q\geq 3p$ is indeed always satisfied, i.e. that $M$ is always homeomorphic to some connected sum of copies of $K3$ and $S^2\times S^2$.

The Euler characteristic and signature of $M$ are given by
\begin{equation}
    \chi = 2+16p + 2q, \qquad \qquad \sigma = -16p.
\end{equation}
In particular, the index (\ref{eqn: index theorem}) of the Dirac operator is equal to $2p$.
Thus, $M$ admits nontrivial harmonic spinors if $p\neq 0$, and consequently, admits no metric of positive scalar curvature if $p\neq 0$.
Furthermore, $2\chi-3|\sigma|=4(1+q-4p)$.
Therefore, $M$ satisfies the Hitchin-Thorpe inequality if and only if $4p\leq q+1$.
It follows that the normalized Ricci flow on $M$ becomes singular if $4p>1+q$, for example if $M$ is homeomorphic to $(r+2)K3 \# rS^2\times S^2$, for any $r\geq 0$.

Equality holds in the Hitchin-Thorpe inequality if and only if $4p=q+1$. Hence the normalized Ricci flow must also become singular if $(p,q)\neq (1,3)$; for example, if $M$ is homeomorphic to $(r+1)K3\#r S^2\times S^2$, for any $r\geq 1$, since these manifolds are not homeomorphic to K3.
The most interesting case, $(p,q)=(1,3)$, namely manifolds homeomorphic to K3, is analyzed below. 

\subsection{Exotic K3s}

Dimension four is special in differential topology: it is the only dimension in which the smooth structure on the Euclidean space $\R^n$ is not unique. That is to say, for any $n\neq 4$, if $M^n$ is a smooth manifold homeomorphic to $\R^n$, then $M^n$ is diffeomorphic to $\R^n$.
On the other hand, $\R^4$ admits uncountably many distinct smooth structures \cite{Ta87}.
Furthermore, dimension four is the only dimension in which a closed, smooth manifold can admit infinitely many distinct smooth structures; in all other dimensions, Kirby-Siebenmann theory implies that a topological manifold admits at most finitely many distinct smooth structures.

An illuminating construction of exotic smooth structures in dimension four is the knot surgery operation on the K3 surface \cite{FS98}; this demonstrates a deep link between knot theory and 4-dimensional differential topology.
To describe this construction, recall that there exists a unique closed, simply-connected, smooth 4-manifold admitting a hyperk\"ahler Riemannian metric. 
The existence of such a metric was proved by Yau and uniqueness follows from the Hitchin-Thorpe inequality \cite{Hi74b}.
The underlying smooth manifold $X$ is called the K3 surface.
It admits a spin structure, and has Euler characteristic $\chi=2$ and signature $\sigma = -16$.

The knot surgery procedure assigns to a knot $K\subset S^3$ a smooth manifold $X_K$ which is homeomorphic, but not diffeomorphic, to $X$. 
The manifold $X_K$ is constructed by replacing the tubular neighborhood $T^2\times D^2\subset X$ of an embedded 2-torus with trivial normal bundle with the complement of a neighborhood $N(K)\subset S^3$ of the knot $K$ times the circle $S^1$
\begin{equation}
    X_K = (X\setminus D^2\times T^2) \cup (S^3\setminus N(K))\times S^1.
\end{equation}
Freedman's theorem implies that $X_K$ is homeomorphic to the standard K3 for any knot $K$.
Remarkably, the Seiberg-Witten invariant of $X_K$ is equivalent to the Alexander polynomial of the knot $K$ \cite{FS98}. 
In particular, for two knots $K,K'$ with distinct Alexander polynomials, the resulting manifolds $X_K$ and $X_{K'}$ have different Seiberg-Witten invariants, and hence are not diffeomorphic.

An important consequence of Theorem \ref{thm: parabolic Hitchin rigidity} is that the normalized Ricci flow must become singular on all exotic K3s, in particular those constructed via knot surgery.

\begin{corollary}\label{cor: K3 singularities}
    Every normalized Ricci flow on an exotic K3 surface becomes singular.
\end{corollary}
%%%%%%%%%%%%%%%%%%%%%%%%%%%%%%%%%%
\appendix
\section{Appendix}

\begin{proposition}
[Weighted Schr\"odinger-Lichnerowicz for twisted Dirac operators]\label{prop: Weighted Schrodinger-Lichnerowicz for twisted Dirac operators}
Let $D_E$ be a twisted Dirac operator on a twisted spin bundle $\Sigma_E=\Sigma M\otimes E\to M$ and let $f:M\to \R$ be a weight function. 
Then the weighted, twisted Dirac operator $D_{E,f}=e^{f/2}D_E e^{-f/2}$ satisfies the weighted Schr\"odinger-Lichnerowicz formula
\begin{equation}
    D_{E,f}^2=-\Delta_{f}+\frac{1}{4}\Scal_f + \mathfrak{R}^E,
\end{equation}
where $\Delta_f=\Delta-\nabla_{\nabla f}$ is the weighted connection Laplacian on $\Sigma_E$ and $\mathfrak{R}^E$ is the endomorphism of $\Sigma_E$ defined locally by
\begin{equation}
    \mathfrak{R}^E(\psi\otimes u)
    =\frac{1}{2}e_i\cdot e_j \cdot \psi \otimes R^{E}_{e_i,e_j}u.
\end{equation}
\end{proposition}

\begin{proof}
By the classical Schr\"odinger-Lichnerowicz formula for twisted Dirac operators [LM, II Eqn.\ (8.23)],
\begin{align}
    D_{E,f}^2
    =e^{f/2}D_E^2e^{-f/2} 
    % \\ \nonumber
    =e^{f/2}\left(-\Delta+\frac{1}{4}\Scal +\mathfrak{R}^E\right)e^{-f/2}
    % \\\nonumber
    =-e^{f/2}\Delta e^{-f/2}+\frac{1}{4}\Scal +\mathfrak{R}^E.
\end{align}
Using the compatibility of the twisted connection with the metric, it is straightforward to check that
\begin{equation}
    e^{f/2}\Delta e^{-f/2} 
    = \Delta_f-\frac{1}{4}(2\Delta f-|\nabla f|^2)
    =\Delta_f -\frac{1}{4}(\Scal_f-\Scal).
\end{equation}
Combined with the previous equation, this concludes the proof.
\end{proof}

%%%%%%%%%%%%%%%%%%%%%%%%%%%%%%%%%%
% \newpage

\newcommand{\Addresses}{{% additional braces for segregating \footnotesize
  \bigskip
  \bigskip
  \small
    \textsc{Department of Mathematics}\par\nopagebreak
    \textsc{Massachusetts Institute of Technology}\par\nopagebreak
    \textsc{77 Massachusetts Avenue} \par\nopagebreak
    \textsc{Cambridge, MA 02139} 
    
    \medskip
    \medskip
    
    \textit{Correspondence to be sent to:} \texttt{juliusbl@mit.edu}

}}

\Addresses

\end{document}